\documentclass[a4paper,UKenglish,cleveref, autoref, thm-restate]{lipics-v2021}

\title{Asymptotic Transfer in Critical Recursive Composition Schemes}
\author{Michael Drmota}{TU Wien, Institute of Discrete Mathematics and Geometry, Wiedner Hauptstrasse 8-10, A1040 Vienna, Austria}{michael.drmota@tuwien.ac.at}{https://orcid.org/0000-0002-6876-6569}{}
\author{Zéphyr Salvy}{TU Wien, Institute of Discrete Mathematics and Geometry, Wiedner Hauptstrasse 8-10, A1040 Vienna, Austria}{zephyr.salvy@tuwien.ac.at}{}{}

\authorrunning{M. Drmota and Z. Salvy} 

\Copyright{Michael Drmota and Zéphyr Salvy} 


\funding{This research was supported by the Austrian Science Fund (FWF) [10.55776/F1002].}


\hideLIPIcs
\nolinenumbers

\EventEditors{John Q. Open and Joan R. Access}
\EventNoEds{2}
\EventLongTitle{42nd Conference on Very Important Topics (CVIT 2016)}
\EventShortTitle{CVIT 2016}
\EventAcronym{CVIT}
\EventYear{2016}
\EventDate{December 24--27, 2016}
\EventLocation{Little Whinging, United Kingdom}
\EventLogo{}
\SeriesVolume{42}
\ArticleNo{23}

\renewcommand{\(}{\left(}
\renewcommand{\)}{\right)}

\newcommand{\twoc}{\texorpdfstring{$2$}{2}-connected}
\newcommand{\nnsi}{necessarily non-self-intersecting}

\newcommand{\ellgon}{pure \texorpdfstring{$\ell$}{ell}-gon}
\newcommand{\ellgons}{\ellgon s}

\newcommand{\p}{\mathfrak{p}}
\newcommand{\m}{\mathfrak{m}}
\renewcommand{\b}{\mathfrak{b}}

\newcommand{\withroot}{^{(r)}}
\newcommand{\subfamily}{_L}

\newcommand{\xface}{x_{\ell}}
\newcommand{\xfacer}{\xface\withroot}
\newcommand{\x}{\mathbf{x}\subfamily}
\newcommand{\xr}{\x\withroot}
\newcommand{\X}{\mathbf{X}}

\newcommand{\xellgon}{\widehat{\xface}}
\newcommand{\xellgonr}{\xellgon\withroot}
\newcommand{\xallgons}{\widehat{\x}}
\newcommand{\xallgonsr}{\xallgons\withroot}

\newcommand{\xpatt}{x_{\p}}

\crefname{equation}{}{Equations} 
\Crefname{equation}{Equation}{Equations}

\begin{document}

\maketitle

\begin{abstract}
The composition $\mathcal{F} \circ \mathcal{G}$ of two combinatorial classes $\mathcal{F}$ and $\mathcal{G}$ is a standard combinatorial construction and translates into the composition $F(G(z))$ of their corresponding counting generating functions. Such a composition is called critical if $G(\rho_G) = \rho_F$, where $\rho_F$ and $\rho_G$ denote the corresponding radii of convergences of $F$ and $G$, respectively. In this case, both the singular behaviours of $F$ and $G$ influence that of $F\circ G$. Such critical decomposition schemes appear quite frequently in the context of map enumeration. For example by using the block-decomposition one has $M(z) = B\(z\(1+M(z)\)^2\)$ and $\rho_B = \rho_M \(1+M(\rho_M)\)^2$, where $M(z)$ denotes the generating series of all rooted planar maps and $B(y)$ the generating series of $2$-connected rooted planar maps. This can be extended to multivariate generating functions by taking several statistics into account, for example face counts. Since critical composition schemes show (usually) a condensation phenomenon -- in the above situation this means that there is \emph{giant} $2$-connected  block of linear size and linearly many small blocks -- it is very plausible that statistical properties on $2$-connected maps transfer to corresponding properties of all maps and back. The purpose of the present paper is to make this precise on the level of the singular structure of the corresponding multivariate generating functions. In particular we show that \emph{moving $3/2$-singularities} transfer. Since such kind of singularities are closely related to central limit theorems of the corresponding statistics this methods provides also a kind of transfer of central limit theorems. Actually this method is quite flexible and is applied to a variety of face and pattern counting statistics in map enumeration.
\end{abstract}

\section{Introduction}
A natural operation in combinatorics is to express a combinatorial class as a composition of two (potentially simpler) combinatorial classes:
\[\mathcal{H} = \mathcal{F} \circ \mathcal{G};\]
where the $\circ$ serves to express that every atom of an element of $\mathcal{F}$ is replaced by an element of $\mathcal{G}$. Such expression is called a \emph{composition scheme}. For $\mathcal{A}$ a combinatorial class, write $A(x)$ for its generating function and $\rho_A$ for the radius of convergence of $A(x)$. It is known that for a composition scheme, the behaviour of the class $\mathcal{H}$ depends on the relative position of $G(\rho_G)$ and $\rho_F$ (see, \emph{e.g.}, \cite[\S VI.9]{flajolet_sedgewick_2009}). Such a composition scheme
is called  \emph{critical} if  $G(\rho_G) = \rho_F$. In this case, both the singular behaviours of $F$ and $G$ influence that of $H$.

In this paper we consider recursive composition schemes, that is, we have
$\mathcal{G} = \mathcal{F} \circ \mathcal{G}$, such that 
the composition $\mathcal{F} \circ \mathcal{G}$ is critical.
Such composition schemes are very present in the map world: for example, consider
rooted planar maps. By using the block-decomposition starting at the block that contains the root edge it was already obtained by Tutte \cite{tutte_1963} that
\[M(z) = B\(z\(1+M(z)\)^2\)\qquad\text{and}\qquad \rho_B = \rho_M \(1+M(\rho_M)\)^2,\]
where $M(z)$ denotes the generating series of all rooted planar maps and $B(y)$ the generating series of $2$-connected rooted planar maps (see \cref{subsec:defin-maps} for precise definitions).
In 2001, Banderier, Flajolet, Schaeffer and Soria listed eleven such composition schemes for maps and proposed to study them all at once using the properties of the composition schemes \cite{airy}.

In the context of a critical composition scheme for maps, a \emph{condensation} phenomenon is generally present: a macroscopic block concentrates a linear portion of the mass, while the rest is concentrated in a linear number of smaller blocks \cite{airy}.
Thus, it would be natural to transfer proper statistics and/or shape parameters
from $2$-connected maps to connected maps and vice versa. This has been done, for example, for the degree distribution \cite{zbMATH06297799}. Since the degree distribution is directly related to the asymptotic behaviour of the expected numbers of vertices (or faces) of given degree
these first order statistics transfer, too. It seems, however, that a second order 
transfer, for example a transfer of a central limit theorem, is not that direct.
Suppose that $X_n^{(d)}$ denotes the random number of faces of degree $d$ in 
random planar maps with $n$ edges, and $Y_n^{(d)}$ the random number of faces of degree $d$ in $2$-connected random planar maps with $n$ edges. Then it is known that
both, $X_n^{(d)}$ and $Y_n^{(d)}$ satisfy central limit theorems with asymptotically 
linear expected values and variances, see \cite{DP13}, however, there is -- up to now -- no
direct transfer of these results (although this is possible for the 
asymptotics of the expected values $\mathbb{E} X_n^{(d)}$ and $\mathbb{E} Y_n^{(d)}$).
Such a transfer result would be even more interesting between $2$-connected and
$3$-connected planar maps (by using, for example, the Tutte decomposition). 
At the moment it is fully open, whether the numbers $Z_n^{(d)}$ of faces 
of degree in $3$-connected planar maps satisfy a central limit theorem. 

The purpose of this article is to present an approach of such a transfer on the level
of the singular structure of the corresponding bivariate generating functions
$M(z,x)$ and $B(z,x)$ where the second variable $x$ takes care of the parameter 
of interest (for example, the number of faces of degree $\ell$ that are \ellgons) and
when it is possible to express $M(z,x)$ as a function involving $B(z,x)$ composed
with $M(z,x)$ (see also \cref{prop:M1-M4-joint-ellgons} with a slightly different
notation: $M = M_1$, $B = M_4$) 
\begin{equation}\label{eqMBzxrel}
{M}(z,x) = {B}\(z\(1+{M}(z,x)\)^2,\frac{\(1+{M}(z,x)\)^\ell-1 + x}{\(1+{M}(z,x)\)^\ell}\).
\end{equation}
It turns out that $B(z,x)$ has a (so-called) \emph{moving $3/2$-singularity} if
and only if $M(z,x)$ has a moving $3/2$-singularity. It is well known that
moving $3/2$-singularities (and similarly moving singularities) are
closely related to a central limit of the corresponding random variable (that is
encoded by the variable $x$), see \cite{Drmota09}, this methods provide also a kind of 
\emph{ transfer of central limit theorems}.

This article is also part of ongoing work in the study of maps, which aims to understand how the number of pattern occurrences behave in (subfamilies of) random planar maps. 
(The simplest pattern is a face of given degree. Thus, pattern counts generalize
face counts of given degree). It is expected that for most ``natural'' patterns, the number of occurrences in a map with $n$ edges follows a central limit theorem  as $n\to\infty$. Drmota and Stufler showed that in planar maps sampled from a (regular critical) Boltzmann distribution, the expected number of pattern occurrences of a given map is asymptotically linear when the number of edges goes to infinity \cite{DS19}. However, the study of the other moments of the distribution requires careful manipulations. For face counts, it has been shown for the classes of all maps and $2$-connected maps \cite{DP13} and for simple maps \cite{Yu19}. In the case of general maps, a joint central limit theorem for all face degrees even holds \cite{CDK19}. Moreover, a 
central limit theorem (CLT) holds for patterns which are \emph{\ellgons} \cite{Yu19} and generally for all patterns with a simple boundary \cite{DHW25}. The present paper extends the list 
of CLT's of that kind.

\begin{table}
\begin{tabular}{l | l l}
Family & Joint CLT & CLT for $\sim$ patterns \\
\hline
Loopless maps & \ellgons{} \& faces & nnsi loopless with simple boundary \\
(Bip) bridgeless maps & \ellgons{} \& faces & nnsi (bip) bridgeless  with simple boundary \\
(Bip) $2$-connected maps & \ellgons{} (= faces) & nnsi (bip) $2$-connected\\
(Bip) simple maps & \ellgons{} \& faces &\\
(Bip) $2$-connected simple maps &\ellgons{} (= faces) &
\end{tabular}
\caption{CLT Results of this article, where ``bip'' abbreviates ``bipartite'' and ``nnsi'' ``\nnsi''. The first line reads ``In the family of loopless maps, we show a joint CLT for \ellgons{} and for face counts, and we show a CLT for \nnsi{} loopless patterns with simple boundary.''}
\label{tab:results}
\end{table}

\subparagraph{Results} Our results are of two natures. First, we show that in several combinatorial 
composition schemes it is also possible to keep track of the statistics of proper parameters like
face counts (resulting in relations of the kind \cref{eqMBzxrel} for corresponding
bi- (and multi-)variate generating series. Secondly we provide a quite general method
how (moving) singularities can be transferred in critical recursive composition schemes.
The corresponding obtained CLT's are summarised in \cref{tab:results}.
(It should be noted that CLT for face counts in $2$-connected maps was already obtained 
by Drmota and Panagiotou \cite{DP13}, and a CLT for \nnsi{} $2$-connected patterns in $2$-connected maps has been obtained by a different method by Hainzl.)

\subparagraph{Plan of the paper} In \cref{sec:defin-notation}, we recall definitions and known results as well as introduce notation necessary for the rest of our work. Then, in \cref{sec:combinatorics}, we use combinatorics to write generating function equations for maps where face (and pattern) occurrences are taken into account. In \cref{sec:transfer-sg-analysis}, we express our result for transfer of singular behaviours, which we apply in \cref{sec:applications} to the equations obtained in \cref{sec:combinatorics}.

\section{Preliminaries}
\label{sec:defin-notation}
\subsection{Singularity notions and a related central limit theorem}
We say that an analytic function $f(z)$ has a \emph{$3/2$-singularity at $z_0$} if $f(z)$ has a local representation of the form
\begin{equation}\label{eq32sing}
    f(z) = g(z) + h(z) (z-z_0)^{3/2},
\end{equation}
where $g(z)$ and $h(z)$ are functions that are analytic at $z=z_0$ with $h(z_0) \ne 0$. We say that an analytic function $f(z,x)$ has a \emph{moving $3/2$-singularity} if $f(z,1)$ has a $3/2$-singularity at some $z_0>0$ and, for $x$ close to $1$, $f(z,x)$ has a local representation of the form
\begin{equation}\label{eq32sing2}
f(z,x) = g(z,x) + h(z,x) (z - \rho(x))^{3/2}
\end{equation}
with $\rho(x)$, $g(z,x)$ and $h(z,x)$ analytic at $x=1$, $\rho(1) = z_0$, and, close to $z_0$, it holds that $g(z,1) = g(z)$ and $h(z,1) = h(z)$. Then, for $x$ close to $1$, $f(z,x)$ also has the following local representation
\[f(z,x) = a_0(x) + a_2(x)\(z - \rho(x)\) + a_3(x)\(z - \rho(x)\)^{3/2} + O\(\(z - \rho(x)\)^{2}\)\]
where the functions $a_j(x)$ are analytic at $x = 1$. This notion of moving $3/2$-singularity can be extended to $f(z,\x)$ where $\x$ is a vector of parameters.

In the context of combinatorial classes $\mathcal{F}$, $f(z,x)$ usually denotes the 
generating series $f(z,x) = \sum_{n,k} f_{n,k} z^n x^k$, where $f_{n,k}$ denotes the number
of objects in $\mathcal{F}$ of size $n$, where another parameter (for example, the number
of faces of given degree) has value $k$. Clearly we have $f_{n,k} \ge 0$ and if we assume that
every object in $\mathcal{F}$ of size $n$ is equally likely then the considered parameter
can be viewed as a random variable $X_n$ with 
\[
\mathbb{P}[X_n = k] = \frac{f_{n,k}}{f_n} = \frac{[z^n x^k]\, f(z,x) }{[z^n]\, f(z,1)}.
\]

\begin{proposition}
\label{Pro0}
Let $f(z,x)$ be the generating function of a combinatorial class $\mathcal{F}$ with $z$ marking the size of the object and $x$ the number of occurrences of a parameter. Then, if $f(z,x)$ has a moving $3/2$-singularity
of the form \cref{eq32sing2} with $\rho'(1) \ne 0$, and if, in a neighbourhood of $x=1$, $\rho(x)$ is the unique dominant singularity of $f(z,x)$ for $|z|= |\rho(x)|$, actually if $f(z,x)$ can be
analytically continued for $|z| = |\rho(x)|$ and $x\ne 1$, then 
$X_n$ follows a central limit theorem:
\[
\frac{X_n - \mathbb{E}\, X_n}{\sqrt{n} } \to N(0,\sigma^2),
\]
where $\mathbb{E}\, X_n \sim \mu n$ and $\mathbb{V}{\rm ar}\, X_n \sim \sigma^2 n$ 
with $\mu =  - {\rho'(1)}/{\rho(1)}$ and $\sigma^2 \ge 0$.
\end{proposition}
This can be extended a (possibly infinite) sequence of parameters, and gives a joint Central Limit Theorem. See \cite[\S 2]{Drmota09} for the proof in the case of a \emph{moving $1/2$-singularity}, which extends to the $3/2$ case.

\subsection{Maps and families of maps}
\label{subsec:defin-maps}

As mentioned above the intention of this paper is to obtain CLT's for several counting
parameters in various map classes by a singularity transfer. We collect here
some basis notions for planar maps.

A \emph{planar map} is a proper embedding of a connected planar finite multigraph into the two-dimensional sphere. A \emph{half-edge} $e$ is an edge (possibly a loop) and an orientation on this edge (each edge yields two half-edges). A \emph{corner} is the angular sector between two consecutive half-edges in anticlockwise order
around a vertex. We will consider rooted planar maps, \emph{i.e.}, planar maps where a half-edge (equivalently, a corner) is distinguished. In the following, \emph{map} is used as a shorthand for ``rooted planar map''. A \emph{face} of a map is a connected component of the complement (in the sphere) of the graph. The \emph{degree} of a face is the number of corners incident to it. The \emph{root face} is the face on the left side of the root when it is a half-edge (equivalently, it is the face incident to the root corner).

The \emph{size} $|\m|$ of a map $\m$ is its number of edges. A map is \emph{loopless} if it does not contain loops. A map is \emph{simple} if it is loopless and has no multiple edges. A \emph{bridge} is an edge whose removal disconnects the map. A map is \emph{bridgeless} if it does not contain bridges. A map is \emph{bipartite} if its vertices can be properly bicoloured in black and white. A map is said to be \emph{separable} if it is possible to partition its edge-set into two non-empty sets $E_1$ and $E_2$ such that there is exactly one vertex (called a \emph{cut vertex}) incident to both a member of $E_1$ and a member of $E_2$. The map is said to be \emph{$2$-connected} otherwise.

\begin{table}
\begin{tabular}{l | c || l | c}
Family & GF & Family & GF\\
\hline
All maps & $M_1(z)$ & Bipartite maps & $B_1(z)$ \\
Loopless/bridgeless maps & $M_2(z)$ & Bipartite simple maps & $B_2(z)$\\
Simple maps & $M_3(z)$ &Bipartite bridgeless maps & $B_3(z)$ \\
$2$-connected maps & $M_4(z)$ &Bipartite $2$-connected maps & $B_4(z)$ \\
$2$-connected simple maps & $M_5(z)$ & Bipartite $2$-connected simple maps & $B_5(z)$\\
\end{tabular}
\caption{Families of maps considered in this article, and their generating functions (GF).}
\label{tab:notation}
\end{table}

We use the same notation as Banderier \emph{et al.}, which we summarise in \cref{tab:notation}. In each case, the variable $z$ marks the size of the map. We do not consider that the vertex map (of size $0$) belongs to any of these families. These generating functions are related by the following equations \cite{airy}:
\begin{alignat}{2}
&M_1(z) = M_2\(\frac{z}{(1-z(1+M_1(z)))^2}\) + z(1+M_1(z))^2 \label{eq:M1-M2}\\
& && \hspace{-0.25\textwidth} B_1(z) = B_3\(\frac{z}{(1-z(1+B_1(z)))^2}\) + z(1+B_1(z))^2 \\
&M_2(z) = M_3(z(1+M_2(z))) && B_1(z) = B_2(z(1+B_1(z)))\\
&M_1(z) = M_4(z(1+M_1(z))^2) && B_1(z) = B_4(z(1+B_1(z))^2) \label{eq:M1-M4}\\
&M_4(z) = M_5(z(1+M_4(z))) && B_4(z) = B_5(z(1+B_4(z)))
\end{alignat}

It is well known that all these generating functions are not only algebraic
but have the same kind of singular behaviour at the radius of convergence 
(see, \emph{e.g.}, \cite{airy}).
\begin{lemma}
All the generating functions $F(z)$ of \cref{tab:notation} have a positive radius of convergence $\rho_F$ which is their unique dominant singularity. Moreover, each $F(z)$ has a $3/2$-singularity at $\rho_F$.
\end{lemma}

\subsection{Face and Pattern counting}

In the following, we use an additional parameter $x_\ell$ to track the 
number of non-root faces of degree $\ell$ in a map.
The corresponding generating functions are then denoted 
by $M_1(z,x_\ell)$ etc. Sometimes we will also count the root face of degree $\ell$.
Then the corresponding counting variable will be denoted by $\xfacer$.

If a face (of degree $\ell$) has the same number of vertices and edges we call it a \emph{\ellgon}.
Then corresponding counting variables are denoted by $\xellgon$ (resp. $\xellgonr$). Notice that, in $2$-connected maps, all faces of degree $\ell$ are \ellgons.

We denote by a \emph{pattern} $\p$ a given map and say that $\p$ \emph{occurs} in a map $\m$ if
$\p$ is a submap of $\m$ (see \cite{Yu19} or \cite{DHW25} for a detailed definition).
For example, a face of degree $\ell$ or a \ellgon{} is a pattern.
If we keep track of a pattern $\p$ then then corresponding counting variable is denoted by $x_\p$.

\begin{figure}
\begin{center}
\includegraphics[width=0.8\textwidth]{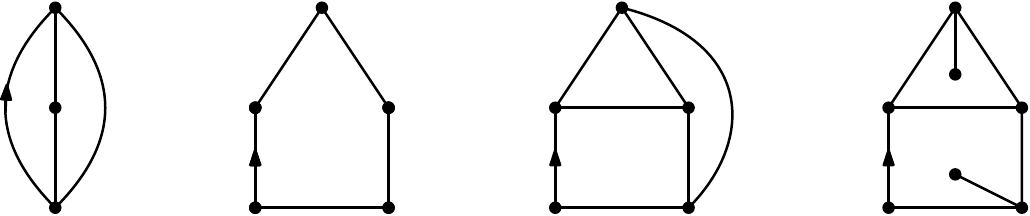}
\end{center}
\caption{Four \nnsi{} patterns.}
\label{fig:nnsi}
\end{figure}

A map is said to have \emph{simple boundary} if its root face is a \ellgon{} for some $\ell\geq 1$. A map $\p$ is said \emph{\nnsi{}} when, for any map $\m$, any interior face of a pattern occurrence of $\p$ in $\m$ cannot be the interior face of another pattern occurrence of $\p$. We also say that $\p$ \emph{cannot self-intersect}. This encompasses in particular the double-triangle pattern and \ellgons. Examples can be found in \cref{fig:nnsi}.

If we keep track of several non-root face counts $\xface$ with $\ell\in L$, where $L$ is a finite
subset of positive integers then we write 
$\x = (\xface)_{\ell \in L}$ for the corresponding vector of counting variables. Similarly, we write $\xr = (\xfacer)_{\ell\in L}$, $\xallgons = (\xellgon)_{\ell \in L}$ and $\xallgonsr = (\xellgonr)_{\ell \in L}$; as well as $\X = (\xface)_{\ell \geq 1}$.

Note that loopless maps of size $n$ and bridgeless maps of size $n$ are in bijection, so both families have the same generating function $M_2(z)$. However, for $\p$ a general pattern, the number of loopless maps of size $n$ with $k$ occurrences of $\p$ is not \emph{a priori} that of bridgeless maps of size $n$ with $k$ occurrences of $\p$. Therefore, when discussing the generating functions where patterns (or faces) are marked, we write $M_{2l}$ (resp. $M_{2b}$) for the generating function of loopless (resp. bridgeless) maps.

From the literate we already know the singular structure 
of some of the corresponding bi- (and multi-)variate generating functions.
Note that the case of $B_1(z,x_{\p})$ is not explicit in the literature,
but can be handled in the same way as $M_1(z,x_{\p})$.
\begin{proposition}
\label{prop:state-art-3-2-sg}
The following generating functions are algebraic and have a moving $3/2$-singularity:
\begin{itemize}
    \item $M_1(z, \xface)$ for $\ell \geq 1$ \cite[Lemma 4]{DP13};
    \item $M_3(z, \xface)$ for $\ell \geq 3$ \cite[\S5.4]{Yu19};
    \item $M_4(z, \xface)$ for $\ell \geq 2$ \cite[\S 6]{DP13};
    \item $M_1(z; \x)$ for any finite family $L$\cite[\S5]{CDK19};
    \item $B_1(z; \x)$ for any finite family $L$ \cite[\S4]{CDK19};
    \item $M_1(z, \xellgon)$ for $\ell \geq 2$ \cite[\S5.2]{Yu19};
    \item $M_1(z,x_{\p})$ for a \nnsi{} pattern $\p$ with simple boundary \cite[Proposition 2.1]{DHW25};
    \item $B_1(z,x_{\p})$ for a \nnsi{} pattern $\p$ with simple boundary, in the same way as \cite[Proposition 2.1]{DHW25}.
\end{itemize}
\end{proposition}
Clearly all mentioned parameters satisfy a CLT by a direct application of \cref{Pro0}.

The new properties that can be obtained with the help of the techniques of this paper
are summarized in the following theorem. The corresponding CLT's are 
listed in \cref{tab:results}.
\begin{theorem}
The following generating functions are algebraic and have a moving $3/2$-singularity:
$M_{2l}(z; \x)$, $M_{2l}(z; x_{\p})$, $M_{2b}(z; \x)$, $M_{2b}(z; x_{\p})$,
$M_{3}(z; \x)$, $M_{3}(z; \xallgons)$,
$M_4(z; \x)$, $M_4(z; \xpatt)$,
$M_5(z; \x)$,
$B_{2}(z; \x)$, $B_{2}(z; \xallgons)$,
$B_{3}(z; \x)$, $B_{3}(z; x_{\p})$,
$B_{4}(z; \x)$, $B_{4}(z; x_{\p})$,
$B_{5}(z; \x)$
where $L$ is any finite set of positive integers and $\p$ is a \nnsi{} pattern with simple boundary.
\end{theorem}
For the proof we use the refined relations for the corresponding multi-variate
generating function from \cref{sec:combinatorics}, the singularity
transfer theorem \cref{LeSing2} plus some additional considerations
that are outlined in \cref{sec:applications}.

The patterns we consider in \cref{sec:combinatorics} are not required to have a simple boundary (but they must be \nnsi). However, since we are transferring a result about patterns with simple boundary, this condition holds in the theorem.

\section{Combinatorics for composition schemes}
\label{sec:combinatorics}
For all the families considered here, we first show how the equation can be obtained for the function $M(z,\xellgon)$. The proof extends readily to $M(z,\xallgons)$ and $M(z,\xpatt)$ for $\p$ a \nnsi{} pattern. We detail the proofs in the case of $2$-connected maps, then we simply give detailed sketches.

\subsection{\twoc{} maps}
\label{subsubsec:M1-M4-nnsi}
As a first case, we consider the decomposition scheme \cref{eq:M1-M4} relating general maps to $2$-connected maps. The decomposition procedure for the map $\m$ consists in finding the $2$-connected maximal submap (\emph{block}) $\b$ of $\m$ containing its root. Then, $\m$ can be obtained from $\b$ by replacing each of its edges by a new edge to which $2$ (possibly) empty maps are attached. Each of these added maps can again be decomposed in the same way.

Notice that a \ellgon{} is a $2$-connected map, so an occurrence of a \ellgon{} in $\m$ is in exactly one block of $\m$. Write $k\geq 0$ for the number of \ellgons{} in the block containing the root. Then, when ``adding back'' the rest of the map, two things can happen:
\begin{itemize}
    \item Either nothing is added in the inner corners of the \ellgon, in which case the \ellgon{} of $\b$ is a \ellgon{} of $\m$;
    \item Or, something is added in a corner, in which case this \ellgon{} of $\b$ is not an occurrence of a \ellgon{} in $\m$.
\end{itemize}
Moreover, a \ellgon{} of $\m$ is either in $\b$ or in another of its block, \emph{i.e.}, a the root block of another of the maps appearing along the decomposition.

Notice that the total number of edges in \ellgons{} in $\b$ might be less than $\ell$ times the number of \ellgons{} because the same edge can belong to several \ellgons{}. Consider the (possibly disconnected) submap of $\b$ obtained by keeping only the \ellgons{}: if $\b$ has $k$ \ellgons{} and $p$ edges in those \ellgons{} then, by double counting, its number of exterior edges is
\begin{equation}
\label{eq:interior-edges-l-gons}
2p - \ell k.
\end{equation}

Let ${M_4}(y,\xellgon,t)$ be the generating function where $t$ counts the total number of edges in \ellgons{}. Notice that ${M_4}(y,\xellgon) = {M_4}(y,\xellgon,1)$. Then, in $M_1(z,\xellgon)$, a map $\m$ with $k$ occurrences of \ellgons{} in its root block $\b$, which have in total $p$ edges contributes
\[\(\frac{[\xellgon^k t^p]{M_4}(y,\xellgon,t)}{y^p}\)\circ \(z\(1+{M_1}(z,\xellgon)\)^2\) z^p \(1+{M_1}(z,\xellgon)\)^{2p-\ell k} \(\(1+{M_1}(z,\xellgon)\)^\ell-1 + \xellgon\)^k\]
where
\begin{itemize}
    \item $\(\frac{[\xellgon^k t^p]{M_4}(y,\xellgon,t)}{y^p}\)\circ \(z\(1+{M_1}(z,\xellgon)\)^2\)$ corresponds to the edges of $\b$ which are not incident to a \ellgon, each of them being endowed with $2$ (possibly) empty maps;
    \item $z^p$ the edges of $\b$ which are in \ellgons{};
    \item $\(1+{M_1}(z,\xellgon)\)^{2p-\ell k}$ corresponds to the outer corners of the \ellgons{}, which can be endowed with any map;
    \item $\(\(1+{M_1}(z,\xellgon)\)^\ell-1 + \xellgon\)^k$ corresponding to the inner corners of the \ellgons{}: if all $\ell$ inner corners of a \ellgon{} of $\b$ are endowed with nothing, then there is a \ellgon{} in $\m$.
\end{itemize}
It remains to sum this term over all possible couples $(k,p)$ to get the following result. This leads to
\begin{equation}
\label{eqM1M4rel}
{M_1}(z,\xellgon) = {M_4}\(z\(1+{M_1}(z,\xellgon)\)^2,\frac{\(1+{M_1}(z,\xellgon)\)^\ell-1 + \xellgon}{\(1+{M_1}(z,\xellgon)\)^\ell}\)
\end{equation}
for all $\ell \ge 1$.

The method used before works when tracking $\ell_1$-gons, \dots, $\ell_n$-gons. Indeed, $\ell_i$-gons do not interact with $\ell_j$-gons for $i\neq j$, and the inner corners of each are disjoint. One simply needs to replace \cref{eq:interior-edges-l-gons} by the following, writing $k_i$ for the number occurrences of $\ell_i$-gons:
$2\ell - \sum_{i=1}^n k_i \ell_i$.

Finally consider a \nnsi{} $2$-connected pattern $\p$. Write $n_\p$ for the size of $\p$, and $e_\p$ for the size of its root face, and $i_\p = n_\p - e_\p$ for its number of interior edges. The number of interior ``sides of edges'' in $\p$ is then $\ell_\p := 2 i_\p + e_\p = 2n_\p - e_\p$. Then, for a (possibly disconnected) map with $k$ occurrences of $\p$ and $p$ edges in those occurrences, it holds, by double counting, that its number of exterior edges is $2p - k\ell_\p$. 
(Notice that for a \ellgon, it holds that $\ell_\p =\ell$ so this is consistent with \cref{eq:interior-edges-l-gons}.) Then, by the same proof as above works.

Summing up, we get the following properties.
\begin{proposition}
\label{prop:M1-M4-joint-ellgons}
For any finite subset $L$ of positive integers it holds that:
\[
{M_1}(z;\xallgons) = {M_4}\(z\(1+{M_1}(z,\xallgons)\)^2;
\left(\frac{\(1+{M_1}(z,\xallgons)\)^\ell-1 + \xellgon}{\(1+{M_1}(z,\xallgons)\)^\ell} \right)_{\ell\in L} \).
\]
Furthermore, for $\p$ a \nnsi{} $2$-connected pattern, it holds that:
\[{
M_1}(z,\xpatt) = {M_4}\(z\(1+{M_1}(z,\xpatt)\)^2,\frac{\(1+{M_1}(z,\xpatt)\)^{\ell_\p}-1 + \xpatt}{\(1+{M_1}(z,\xpatt)\)^{\ell_\p}}\).
\]
\end{proposition}

\subsection{Loopless or bridgeless maps}
\label{subsubsec:M1-M2-nnsi}
Consider the decomposition scheme \cref{eq:M1-M2} relating general maps and loopless maps. In this case, two possible phenomena can happen
\begin{itemize}
    \item The root edge is a loop, and in each of its corner a possibly empty map is attached. This is called the \emph{coreless} case;
    \item The root edge is not a loop. Then consider the \emph{loopless core} $\b$, \emph{i.e.}, the maximal loopless submap containing the root. In each of its corners, a (possibly empty) sequence of loops is attached, each of this loop containing a (possibly empty) general map.
\end{itemize}

A central fact here is that \ellgons{} cannot have loops (except for the $1$-gon, which we do not consider in this case), so an occurrence of a \ellgon{} in $\m$ is in exactly one block of $\m$, and, whenever the composition scheme adds a loop in a face, this face cannot be a \ellgon.
\begin{itemize}
    \item The coreless case simply contributes $z M_1(z,\xellgon)^2$ (because the two root faces of the $M_1$ maps are changed into a face containing a loop, so cannot contribute to the number of \ellgons{}).
    \item When there is a loopless core, all its non-root face \ellgons{} stay that way if and only if each corner of the \ellgon{} is left empty (an empty sequence is attached).
\end{itemize}

As before, this works as well when counting simultaneously simple $\ell_1$-gons, \dots, simple $\ell_n$-gons and when counting a \nnsi{} loopless pattern $\p$ with $\ell_\p$ interior sides of edges.

\begin{proposition}
\label{prop:M1-M2-nnsi}
For any finite subset $L$ of positive integers it holds that:
\[
M_1(z;\xallgons) = M_{2l}\(\frac{z}{\(1-z(1+M_1(z,\xallgons))\)^2};
\left(\frac{\({1-z(1+M_1(z,\xallgons))}\)^{-\ell}-1 + \xellgon}
{\({1-z(1+M_1(z,\xallgons))}\)^{-\ell}}\right)_{\ell\in L}\) + z \(1+M_1(z,\xallgons)\)^2.\]
For $\p$ a \nnsi{} loopless pattern, it holds that:
\[M_1(z,\xpatt) = M_{2l}\(\frac{z}{\(1-z(1+M_1(z,\xpatt))\)^2},\frac{\({1-z(1+M_1(z,\xpatt))}\)^{-\ell_\p}-1 + \xpatt}{\({1-z(1+M_1(z,\xpatt))}\)^{-\ell_\p}}\) + z \(1+M_1(z,\xpatt)\)^2.\]
\end{proposition}

\begin{remark}
\label{rem:M1-M2b-nnsi}
\Cref{prop:M1-M2-nnsi} and the discussion above it still hold when replacing ``loop'' by ``bridge'', ``loopless'' by ``bridgeless'' and $M_{2l}$ by $M_{2c}$, with the exception that $1$-gon can be counted in the decomposition into bridgeless components (and the equation is exactly what one would expect).
\end{remark}

\subsection{(\twoc{}) simple maps}
\label{subsubsec:M2-M3-nnsi}
The simple core of a loopless map is obtained by collapsing maximal multi-edges into a simple edge. In the other direction, a loopless map is obtained from its simple core by attaching a (potentially empty) loopless map to each edge.

When performing this attachment, all the faces of the simple core are preserved. All the interior faces of the attached maps are preserved. For an attached map, the two vertices of the root edge are now attached by an additional edge, creating a new interior face with same degree as the root face of the attached map and the same number of distinct vertices. Therefore, the first equation of \cref{prop:M2-M3-nnsi} holds. (Note that loopless and simple maps do not contain $1$-gons (loops).)

As before, this works as well when counting simultaneously simple $\ell_1$-gons, \dots, simple $\ell_n$-gons. However, contrary to what happens in \cref{subsubsec:M1-M2-nnsi,subsubsec:M1-M4-nnsi}, this does not extend to (\nnsi{}) simple patterns with several faces because a simple pattern can span two simple blocks while staying simple.

\begin{proposition}
\label{prop:M2-M3-nnsi}
For any finite subset $L$ of positive integers, it holds that:
\[M_{2l}(z;\xallgonsr) = M_3(z (1+M_{2l}(z;\xallgonsr));\xallgonsr).\]
The same relation holds when replacing $M_{2l}$ by $M_4$ and $M_3$ by $M_5$.
\end{proposition}

\subsection{Bipartite maps}
Bipartite maps are necessarily loopless, but can have bridges. The discussion for bridgeless patterns in \cref{subsubsec:M1-M2-nnsi} extends to the case where everything considered is bipartite. All the other discussions also extend in the same way.

\begin{proposition}
When considering only bipartite patterns:
\begin{itemize}
    \item \Cref{prop:M1-M4-joint-ellgons} hold when replacing $M_1$ by $B_1$ and $M_4$ by $B_4$;
    \item \Cref{rem:M1-M2b-nnsi} holds when replacing $M_1$ by $B_1$ and $M_{2b}$ by $B_3$;
    \item The first part of \cref{prop:M2-M3-nnsi} holds when replacing $M_{2l}$ by $B_1$ and $M_3$ by $B_2$;
    \item The second part of \cref{prop:M2-M3-nnsi} holds when replacing $M_4$ by $B_4$ and $M_5$ by $B_5$.
\end{itemize}
\end{proposition}

\subsection{Face counts}
As explained before, in $2$-connected maps, all faces of degree $\ell$ are \ellgons, which we have previously treated. We now look at face counts in the other cases.

We start with the loopless case. Let $t$ mark the degree of the root face and consider $M_1(z,t;\X)$. Once again we distinguish according to the existence of a loopless core. The coreless case contributes
\[zt\(x_1 + \sum_{\ell=1}^\infty\([t^\ell] M_1\(z,t;\X\) \cdot \xface\)\) (1 + M_1(z,t;\X)).\]
On the other hand, when there is a core, the degrees of its faces are increased by the sum of sizes of sequences put in the corner of this face. Consider the generating series for sequences of maps, each element in a loop, with the loops lay in the root face:
\[S_l(z,t;\X) = \frac{1}{1-zt\(x_1 + \sum_{\ell=1}^\infty\([t^\ell] M_1\(z,t;\X\) \cdot \xface\)\)}.\]
Then, this case contributes:
\[M_{2l}\(z, S_l(z,t;\X); \xface \leftarrow \sum_{k=0}^{\infty} [t^k] (S(z,t;\X)^\ell ) x_{\ell+k}\).\]

In bridgeless maps, the same reasoning holds, where the coreless case contributes
\[z t^2 \(1+M_1(z,t;\X)\)^2\]
and, for the case with a core, the generating function for the sequences reads
\[S_b(z,t;\X) = \frac{1}{1-zt^2 \(1+M_1(z,t;\X)\)}.\]

\begin{proposition}
\label{prop:M1-M3-faces}
It holds that:
\begin{align*}
M_1(z,t;\X) &= M_{2l}\(z, S_l(z,t;\X); \xface \leftarrow \sum_{k=0}^{\infty} [t^k] (S_l(z,t;\X)^\ell ) x_{\ell+k}\)\\
&\hspace{1cm}+ zt\(x_1 + \sum_{\ell=1}^\infty\([t^\ell] M_1\(z,t;\X\) \cdot \xface\)\) (1 + M_1(z,t;\X));\\
M_1(z,t;\X) &= M_{2b}\(z, S_b(z,t;\X); \xface \leftarrow \sum_{k=0}^{\infty} [t^k] (S_b(z,t;\X)^\ell ) x_{\ell+k}\) + z t^2 \(1+M_1(z,t;\X)\)^2.
\end{align*}
\end{proposition}

Moreover, the discussion of \cref{subsubsec:M2-M3-nnsi} gives the following.
\begin{proposition}
\label{prop:M2-M3-M4-M5-faces}
\Cref{prop:M2-M3-nnsi} holds when $\xallgonsr$ is replaced by $\xr$.
\end{proposition}

Finally, the previous discussions still hold in the bipartite case, with the difference that all faces must have even degrees.
\begin{proposition}
When considering only faces of even degrees:
\begin{itemize}
    \item \Cref{prop:M1-M3-faces} holds when replacing $M_1$ by $B_1$ and $M_{2b}$ by $B_3$;
    \item \Cref{prop:M2-M3-M4-M5-faces} holds when replacing $M_{2l}$ by $B_1$ and $M_3$ by $B_2$.
\end{itemize}
\end{proposition}

\section{Transferring singularity analysis}
\label{sec:transfer-sg-analysis}
We start with a simple inversion theorem for $3/2$-singularities (see \cref{eq32sing}).
\begin{lemma}\label{LeSing1}
    Suppose $f(z)$ has a $3/2$-singularity at $z_0$ of the form 
    \cref{eq32sing} with $g'(z_0)\ne 0$ and $h(z_0) \ne 0$.
    Then $f(z)=u$ can be locally
    inverted and the inverse function has 
    a $3/2$-singularity at $u_0 = g(z_0)$:
    \begin{equation}\label{eqLeSing1}
    z = G(u) + H(u) (u-u_0)^{3/2}
    \end{equation}
    with analytic functions $G(u)$, $H(u)$ that satisfy
    $G(u_0) = z_0$, $G'(u_0) = 1/g'(z_0)$, and 
    $H(u_0) = - h(u_0)/(g'(z_0))^{5/2}$.
\end{lemma}

\begin{proof}
It is very easy to check this property in a formal way. By setting
$Z = z-z_0$ the relation \cref{eq32sing} rewrites (with $u = f(z)$) to
\[
u = g_0 + g_2 Z + g_3 Z^{3/2} + g_4 Z^2 + \cdots,
\]
where $g_0 = g(z_0)$, $g_2 = g'(z_0) \ne 0$, and $g_3 = h(z_0) \ne 0$.
Setting $U = u-g_0$ we also have
\begin{equation}\label{eq32singvar}
U = g_2 Z + g_3 Z^{3/2} + g_4 Z^2 + \cdots
\end{equation}

If we now assume that the inverted function expands in a corresponding
way, then we should have a relation of the form
\[
Z = h_2 U + h_3 U^{3/2} + h_4 U^2 + \cdots
\]
with $h_2\ne 0$ and $h_3\ne 0$.
From that it would follow that
\[
Z^{3/2} = h_2^{3/2} U^{3/2}\left( 1 + \frac{3h_3}{2h_2} U^{1/2} + \cdots \right).
\]
By inserting the representation for $Z$ and $Z^{3/2}$ into \cref{eq32singvar}
and by comparing the coefficients of $U$ and $U^{3/2}$ it follows that
$
g_2 h_2 = 1 \quad \mbox{and} \quad g_2 h_3 + g_3 h_2^{3/2}
$
which implies that
$
h_2 = \frac 1{g_2} \quad \mbox{and} \quad h_3 = -{g_3}/{g_2^{5/2}}.
$

The main difficulty in the proof is to show that there is representation
of the form \cref{eqLeSing1} for the inverse function.
    We set $f(z) = u$ and rewrite \cref{eq32sing} to
    \[
    Q(z,u) = (u- g(z))^2 - h(z)^2 (z-z_0)^3 = 0.
    \]
    At $z=z_0$ and $u = u_0 = g(z_0)$ we have
    \[
    Q(z_0,u_0) = Q_{z}(z_0,u_0) = 0, \quad  Q_{zz}(z_0,u_0) = 2g'(z_0)^2 \ne 0.
    \]
    Hence, by the Weierstrass theorem we can represent $Q(z,u)$ locally as
    \[
    Q(z,u) = K(z,u)( (z-z_0)^2 + h_1(u)(z-z_0) + h_0(u)),
    \]
    where $h_0(u), h_1(u), K(z,u)$ are analytic and satisfy
    $h_0(u_0) = h_1(u_0) = 0$ and $K(z_0,u_0) \ne 0$.
    Hence, if we want to invert the function $u = f(z)$
    we have to solve the quadratic equation
    \[
    (z-z_0)^2 + h_1(u)(z-z_0) + h_0(u) = 0.
    \]
    Clearly, solutions are given by
    \[
    z = z_0 - \frac{h_1(u)}2 \pm \sqrt{\frac{h_1(u)^2}4 - h_0(u)}.
    \]
    It remains to study the term $\frac{h_1(u)^2}4 - h_0(u)$. 
    From the relation
    \[
    (u- g(z))^2 - h(z)^2 (z-z_0)^3 = K(z,u)( (z-z_0)^2 + h_1(u)(z-z_0) + h_0(u))
    \]
    it follows (by setting $z=z_0$) that
    $
    h_0(u) = {(u-u_0)^2}/{K(z_0,u)}
    $
    and that (after differentiating with respect to $z$ and setting $z=z_0$)
    \[
    h_1(u) = - \frac{2(u-u_0)g'(z_0) + K_z(z_0,u)h_0(u)}{K(z_0,u)}.
    \]
    Furthermore, by differentiating twice with respect to $z$ and setting
    $z=z_0$ and $u=u_0$ we obtain
    $
    K(z_0,u_0) = g'(z_0)^2.
    $
    This implies that
    \[
    \frac{h_1(u)^2}4 - h_0(u) = O((u-u_0)^3).  
    \]
    By computing few more derivatives we can also compute $K_u(z_0,u_0)$
    which gives 
    \[
    \frac{h_1(u)^2}4 - h_0(u) = C(u-u_0)^3   + O(u-u_0)^4
    \]
    with
    $
    C = \frac{h(z_0)^2}{ g'(z_0)^5} \ne 0.
    $
    This finally implies that
    \[
    \sqrt{\frac{h_1(u)^2}4 - h_0(u)} = H(u) (u-u_0)^{3/2}
    \]
    for an analytic function $H(u)$ with $H(u_0)\ne 0$ and completes the proof of the theorem.
\end{proof}

\begin{remark} \label{rem:LeSing1-param-x}\Cref{LeSing1} directly extends to functions
that depend (analytically) on a further parameter $x$ (locally around $x=x_0$). 
Suppose that a function $f(z,x)$ has a \emph{moving} $3/2$-singularity of the form \cref{eq32sing2}
with analytic functions $g(z,x)$, $h(z,x)$, $\rho(x)$ with $g_z(x,\rho(x)) \ne 0$.
Then we invert the relation $u = f(z,x)$ to 
\[
z = G(u,x) + H(u,x) (u-\kappa(x))^{3/2}
\]
for proper analytic functions $G(z,x)$, $H(z,x)$ and  $\kappa(x) = g(\rho(x),x)$. In that case, it holds that
\[G(\rho(x),x) = \rho(x),\quad G_u(\rho(x),x) = \frac{1}{g_z(\rho(x),x)}\quad\text{and}\quad H(\rho(x),x) = -\frac{h(\rho(x),x)}{g_z(\rho(x),x)^{5/2}}.\]
\end{remark}

We are now ready to prove a more general transfer property for
$3/2$-singularities.
\begin{theorem}\label{LeSing2}
Suppose that the functions $f_1(z,x)$, $f_2(z,x)$ have $3/2$-singularities
in $z$, where $x$ is an additional analytic variable that varies locally
around $x= x_0$, of the form
\[
      f_1(z,x) = g_1(z,x) + h_1(z,x) (z-\rho(x))^{3/2}, \quad
      f_2(z,x) = g_2(z,x) + h_2(z,x) (z-\rho(x))^{3/2}
\]
with analytic functions $g_1(z,x)$, $h_1(z,x)$, $g_2(z,x)$, $h_2(z,x)$, $\rho(x)$
that satisfy:
\[
g_{1,z}(\rho(x),x) \ne 0, \quad h_1(\rho(x),x)\ne 0, \quad
g_{2,z}(\rho(x),x) \ne 0, \quad h_2(\rho(x),x)\ne 0, 
\]
\[
h_2(\rho(x),x)g_{1,z}(\rho(x),x) \ne h_1(\rho(x),x)g_{2,z}(\rho(x),x).
\]
Then the relation $u = f_1(z,x)$, $v = f_2(z,x)$ can be locally inverted 
such that there are corresponding $3/2$-singularities:
\[
      z = G_1(u,v) + H_1(u,v) (v-R(u))^{3/2}, \quad
      x = G_2(u,v) + H_2(u,v) (v-R(u))^{3/2}
\]
with proper analytic functions $G_1(u,v)$, $H_1(u,v) $,  $G_2(u,v)$, $H_2(u,v)$,
$R(u)$ that satisfy corresponding conditions:
\[
G_{1,u}(u,R(u)) \ne 0, \quad H_1(u,R(u))\ne 0, \quad
G_{2,u}(u,R(u)) \ne 0, \quad H_2(u,R(u))\ne 0, 
\]
\[
H_2(u,R(u))G_{1,z}(u,R(u)) \ne H_1(u,R(u))G_{2,z}(u,R(u)).
\]
The function $R(u)$ is given by
$R(u) = g_2(\rho(\beta(u)),\beta(u))$, where $\beta(u)$ is the inverse 
function of $\alpha(x) = g_1(\rho(x),x)$.
\end{theorem}

\begin{remark}\label{rem:LeSing2-more-variables}
\Cref{LeSing2} extends to more than $2$ variables. For the sake of brevity we
do not make this explicit. The proof can be done, for example, by an induction
on the number of variables.
\end{remark}

\begin{proof}
    The idea of the proof is to apply \cref{LeSing1} twice, however, in
    the more general form with an additional analytic parameter (as mentioned in \cref{rem:LeSing1-param-x}
    above).

    Let's start with the relation
    \[
    u = f_1(z,x) = g_1(z,x) + h_1(z,x) (z-\rho(x))^{3/2}
    \]
    that translates to 
    \[
    z = J_1(u,x) + J_2(u,x) (u- \alpha(x))^{3/2},
    \]
    with $\alpha(x) = g_1(\rho(x),x)$. Furthermore we have
    \[
    J_1(\alpha(x),x) = \rho(x), \quad 
    J_{1,u}(\alpha(x),x) = \frac 1{g_{1,z}(\rho(x),x)} \ne 0, \quad
    J_2(\alpha(x),x) = - \frac{h_1(\rho(x),x)}{g_{1,z}(\rho(x),x)^{5/2}} \ne 0.
    \]
    This relation is now used to substitute $z$ in the 
    relation $v = f_2(z,x) = g_2(z,x) + h_2(z,x) (z-\rho(x))^{3/2}$
    which leads to
    \begin{align*}
        v &= f_2(z,x) = g_2(z,x) + h_2(z,x) (z-\rho(x))^{3/2} \\
        &= g_2(J_1(u,x) + J_2(u,x) (u- \alpha(x))^{3/2},x)  \\
        &+ h_2(J_1(u,x) + J_2(u,x) (u- \alpha(x))^{3/2},x) 
        \left(  J_1(u,x) + J_2(u,x) (u- \alpha(x))^{3/2} - \rho(x) \right)^{3/2}.
    \end{align*}
    Clearly, we can represent $g_2(J_1(u,x) + J_2(u,x) (u- \alpha(x))^{3/2},x)$ by
    applying a local Taylor expansion as
    \[
    g_2(J_1(u,x) + J_2(u,x) (u- \alpha(x))^{3/2},x) 
    = L_1(u,x) + L_2(u,x)  (u- \alpha(x))^{3/2}
    \]
    with analytic functions $L_1(u,x)$ and $L_2(u,x)$, where $L_1(\alpha(x),x) = g_2(\rho(x),x)$,
    \[L_{1,u}(\alpha(x), x) = \frac{g_{2,z}(\rho(x),x)}{g_{1,z}(\rho(x),x)}\quad\text{and}\quad L_2(\alpha(x),x) = g_{2,z}(\rho(x),x) J_2(\rho(x),x).\] Similarly we obtain
    \[
    h_2(J_1(u,x) + J_2(u,x) (u- \alpha(x))^{3/2},x) 
    = L_3(u,x) + L_4(u,x)  (u- \alpha(x))^{3/2}
    \]
    with analytic functions $L_3(u,x)$ and $L_4(u,x)$, where $L_3(\alpha(x),x) = h_2(\rho(x),x)$. Finally, since
    $J_1(\alpha(x),x) = \rho(x)$
    we have
    \begin{align*}
       \left(  J_1(u,x) + J_2(u,x) (u- \alpha(x))^{3/2} - \rho(x) \right)^{3/2}
       &= \left( \tilde J(u,x)(u-\alpha(x)) +  J_2(u,x) (u- \alpha(x))^{3/2}         \right)^{3/2} \\
       &= L_5(u,x) (u-\alpha(x))^{3/2} + L_6(u,x)(u-\alpha(x))^{2}
    \end{align*} 
    with analytic functions $\tilde J(u,x)$, $L_5(u,x)$, $L_6(u,x)$
    that satisfy $\tilde J(\alpha(x),x) = J_{1,u}(\alpha(x),x) \ne 0$ and 
    $L_5(\alpha(x),x) = J_{1,u}(\alpha(x),x)^{3/2} = g_{1,z}(\rho(x),x)^{-3/2} \ne 0$. Summing up, this leads to an expansion of the form
    \[
    v = J_3(u,x) + J_4(u,x) (u- \alpha(x))^{3/2}
    \]
    with proper analytic functions $J_3(u,x), J_4(u,x)$ that satisfy
    \[
    J_3(\alpha(x),x) = L_1(\alpha(x),x) = g_2(\rho(x),x), \quad
    J_{3,u}(\alpha(x),x) = L_{1,u}(\alpha(x),x) = \frac{g_{2,z}(\rho(x),x)}{g_{1,z}(\rho(x),x)},
    \]
    \[
    J_4(\alpha(x),x) = L_2(\alpha(x),x) + L_3(\alpha(x),x) L_5(\alpha(x),x) = \frac{h_2(\rho(x),x)g_{1,z}(\rho(x),x)-h_1(\rho(x),x) g_{2,z}(\rho(x),x)}{g_{1,z}(\rho(x),x)^{5/2}}.
    \]

    In the next step we apply the Weierstrass preparation theorem to 
    represent $u- \alpha(x)$  (locally around $u=u_0 = f_1(\rho(x_0),x_0)$ and $x = x_0$)  as 
    \[
    u- \alpha(x) = K(u,x)( x - \beta(u)),
    \]
    where $\beta(u)$ is the inverse function of $\alpha(x)$ and $K(u,x)$
    is an analytic function with $K(u,x)\ne 0$. This leads to representation
    of the kind
    \[
        z =  J_1(u,x) + \overline J_2(u,x) (x-\beta(u))^{3/2}, \quad
        v = J_3(u,x) + \overline J_4(u,x) (x- \beta(u))^{3/2}
    \]
    with $\overline J_2(u,x) = J_2(u,x) K(u,x)^{3/2}$ and
    $\overline J_4(u,x) = J_4(u,x) K(u,x)^{3/2}$. 
    Note that this change of variables does not change the properties 
    that the coefficients of $(x-\beta(u))$ and $(x-\beta(u))^{3/2}$ are non-zero.
    Furthermore, since 
    \[
    \frac{1}{g_{1,z}}\left( \frac{h_2g_{1,z}-h_1g_{2,z}}{g_{1,z}^{5/2}} \right)
    + \frac{g_{2,z}}{g_{1,z}} \frac{h_1}{g_{1,z}^{5/2}} = \frac{h_2}{g_{1,z}^{5/2}} \ne 0
    \]
    it also follows that, in $(\alpha(x),x)$:
    \[
    J_4 J_{1,u} - J_2 J_{3,u} \ne 0.
    \]

    In a second step we invert the relation 
    $v = J_3(u,x) + \overline J_4(u,x) (x- \beta(u))^{3/2}$ 
    (with the parametrized version of 
    \cref{LeSing1}) to
    \[
    x = G_2(u,v) + H_2(u,v) (v-R(u))^{3/2}
    \]
    with proper analytic functions $G_2(u,v)$, $H_2(u,v)$, $R(u)$ and 
    use this relation to substitute $x$ into the relation
    $z =  J_1(u,x) + \overline J_2(u,x) (x-\beta(u))^{3/2}$ which leads
    in completely the same way as above to the representation
    \[
    z = G_1(u,v) + H_1(u,v) (v-R(u))^{3/2}
    \]
    with proper analytic functions $G_1(u,v)$, $H_1(u,v)$ that satisfy
    the proposed properties. This completes the proof of the lemma.
\end{proof}

\section{Applications}
\label{sec:applications}
\subsection{\twoc{} planar maps}

Let's start with the generating function ${M_1}(z,x)$ of planar maps, where
$z$ counts the edges and $x$ the number of \ellgons{} which have distinct vertices and edges. We know from \cref{prop:state-art-3-2-sg} that ${M_1}(z,x)$ has a $3/2$-singularity that moves in $x$ that we represent as 
\[
{M_1}(z,x) = g(z,x) + h(z,x) (z-\rho(x))^{3/2}.
\]
Note that the singular part is usually written as 
$(1- z/\rho(x))^{3/2}$ but this is reflected in a minor change in 
the function $h(z,x)$.

From \cref{prop:M1-M4-joint-ellgons} we know that ${M_1}(z,x)$ satisfies the relation \cref{eqM1M4rel}
where ${M_4}(u,v)$ is the corresponding generating function for
2-connected planar maps. We now consider the relations
\[
    u = z{M_1}(z,x)^2, \quad
    v = \frac{{M_1}(z,x)^\ell-1 + x}{{M_1}(z,x)^\ell}.
\]
By \cref{eqM1M4rel} it follows that these two relations satisfy the
assumptions of \cref{LeSing2}. Thus, it follows that these two
relations can be inverted in order to obtain relations of the form
\[
      z = G_1(u,v) + H_1(u,v) (v-R(u))^{3/2}, \quad
      x = G_2(u,v) + H_2(u,v) (v-R(u))^{3/2}.
\]
Since ${M_4}(u,v) = {M_1}(z,x) = \sqrt{ \frac uz}$
we also get a relation of the form
\[
{M_4}(u,v) = G(u,v) + H(u,v) (v-R(u))^{3/2}
\]
for proper analytic functions $G(u,v)$, $H(u,v)$ as proposed.

The case of several \ellgons{} counts can be handled in
the same way. One just need an extension of \cref{LeSing2} to
several variables (see \cref{rem:LeSing2-more-variables}).

The case of non-intersecting pattern counts is also immediate.

\subsection{Face counts in loopless and  bridgeless maps}

For the case of face counts in loopless or bridgeless maps   
(compare with Proposition~\ref{prop:M1-M3-faces})
we have to proceed in a slightly different way and for the sake of brevity
we only consider the case $\ell = 2$ in brideless maps, that is, we consider the equation
\begin{align*}
{M_1}(z,t;x_1,x_2) 
&= {M_{2b}}\Biggl( z, (1-zt^2{M_1}(z,t;x_1,x_2))^{-1}; \\
& \tilde  x_1 \leftarrow 
x_1 + [t](1-zt^2{M_1}(z,t;x_1,x_2))^{-1} x_2  +\sum_{k=2}^{\infty} [t^k] (1-zt^2{M_1}(z,t;x_1,x_2))^{-1}, \\
&  \tilde x_2 \leftarrow 
x_2 + \sum_{k=1}^{\infty} [t^k] (1-zt^2{M_1}(z,t;x_1,x_2))^{-2} \biggr)
+ zt^2(1+{M_1}(z,t;x_1,x_2))^2 .   
\end{align*}
We recall that ${M_1}(z,t;x_1,x_2)$ has a local representation of the form
(see \cite{CDK19})
\[
{M_1}(z,t;x_1,x_2) = g(z,t,x_1,x_2) + h(z,t,x_1,x_2) (z- \rho(x_1,x_2))^{3/2},
\]
where the functions $g(z,t,x_1,x_2)$, $h(z,t,x_1x_2)$, and $\rho(x_1,x_2)$ are analytic. 
In particular, it follows that all functions that are substituted in $M_{2b}$ have a
$3/2$-singularity:
\begin{align*}
    x_1 + [t](1-zt^2{M_1}(z,t;x_1,x_2))^{-1} x_2 &= g_1(z, x_1,x_2)+ h_1(z,x_1,x_2) (z- \rho(x_1,x_2))^{3/2}, \\
    \sum_{k=2}^{\infty} [t^k] (1-zt^2{M_1}(z,t;x_1,x_2))^{-1} &= 
    g_2(z, x_1,x_2)+ h_2(z,x_1,x_2) (z- \rho(x_1,x_2))^{3/2}, \\
    x_2 +\sum_{k=1}^{\infty} [t^k] (1-zt^2{M_1}(z,t;x_1,x_2))^{-2}  &= g_3(z, x_1,x_2)+ h_3(z,x_1,x_2) (z- \rho(x_1,x_2))^{3/2}.
\end{align*}
We now apply the substitution 
\begin{align*}
\tilde t &= (1-zt^2{M_1}(z,t;x_1,x_2))^{-1} = g_0(z,t,x_1,x_2) + h_0(z,t,x_1,x_2) (z- \rho(x_1,x_2))^{3/2}, \\
\tilde x_1 &= (g_1(z, x_1,x_2)+g_2(z,x_1,x_2))+ (h_1(z,x_1,x_2)+h_2(z,x_1,x_2)) (z- \rho(x_1,x_2))^{3/2}, \\
\tilde x_2 &= g_3(z, x_1,x_2)+ h_3(z,x_1,x_2) (z- \rho(x_1,x_2))^{3/2}.
\end{align*}
We first invert the last two equations and obtain
\begin{equation}
\label{eqx2tildex2}
x_1 = \overline G_1(z,\tilde x_1,\tilde x_2) + \overline H_1(z,\tilde x_1, \tilde x_2) 
(z- \overline \rho(\tilde x_2))^{3/2}, \quad
x_2 = \overline G_2(z,\tilde x_1,\tilde x_2) + \overline H_2(z,\tilde x_1, \tilde x_2) 
(z- \overline \rho(\tilde x_2))^{3/2}
\end{equation}
Secondly we apply a regular inversion to the first equation and obtain
\[
t = \overline g( z,\tilde t, x_1,x_2) + \overline h( z,\tilde t, x_1,x_2) (z- \rho(x_1,x_2))^{3/2}
\]
By substituting $x_1$  and $x_2$ with the 
help of \cref{eqx2tildex2} we get a representation of the form
\begin{equation}
\label{eqttildet}
t = \overline g_2( z,\tilde t, \tilde x_1,\tilde x_2) + \overline h_2( z,\tilde t, \tilde x_1, \tilde x_2) 
(z- \overline \rho(\tilde x_1, \tilde x_2))^{3/2}.    
\end{equation}
Finally, by substituting $t$, $x_1$, and $x_2$ with the help of
\cref{eqttildet} and \cref{eqx2tildex2} we obtain
\begin{align*}
M_{2b}(z,\tilde t; \tilde x_1,\tilde x_2) &= M_1(z,t;x_1,x_2) - zt^2 ( 1 + M_1(z,t;x_1,x_2))^2\\
&= \tilde G(z,\tilde t, \tilde x_1, \tilde x_2) +  \tilde H(z,\tilde t, \tilde x_1, \tilde x_2) 
(z- \overline \rho(\tilde x_1, \tilde x_2))^{3/2}
\end{align*}
with proper analytic functions $\tilde G(z,\tilde t, \tilde x_1, \tilde x_2), 
\tilde H(z,\tilde t, \tilde x_1, \tilde x_2)$.

\subsection{Other cases}

All the other cases (involving loopless, bridgeless, simple, and bipartite maps) 
can be handled by the same principles.

\bibliography{references.bib}

\end{document}